\documentclass[11pt,reqno]{amsart}
\usepackage{amscd,amsmath,amsopn,amssymb,amsthm,multicol}
\usepackage{tikz,subdepth,anysize,verbatim,ifthen,xargs,colortbl,float}
\usepackage{longtable,mathtools,hyperref,chemarrow}
\usepackage[english]{babel}\usetikzlibrary{positioning}
\everymath=\expandafter{\the\everymath\displaystyle}

\theoremstyle{plain}
\newtheorem{theorem}{Theorem}
\newtheorem{prop}{Proposition}
\newtheorem{lemma}{Lemma}
\newtheorem{cor}{Corollary}
\newtheorem{rk}{Remark}
\theoremstyle{definition}

\newcommand\com[1]{}
\newcommand\C{{\mathbb C}}

\newcommand\g{{\mathfrak g}}
\newcommand\h{\mathfrak{h}}

\newcommand\m{{\frak m}}

\newcommand\op[1]{\mathop{\rm #1}\nolimits}
\newcommand\p{\partial}

\newcommand\R{{\mathbb R}}
\newcommand\Z{{\mathbb Z}}

\begin{document}

\title{On CR manifolds of CR dimension 1}

\author{Boris Kruglikov}
\address{Department of Mathematics and Statistics, UiT the Arctic University of Norway, Troms\o\ 9037, Norway.
\ E-mail: {\tt boris.kruglikov@uit.no}. }

 \begin{abstract}
We classify all maximal symmetry models of CR dimension 1, depending on their Bloom-Graham and Tanaka types, 
give coordinate realization to some of those models and prove a general extension principle.
 \end{abstract}

\maketitle

\section{Formulation of the Problem and the Results}\label{S1}

Consider a generic real submanifold $M^m\subset\C^n$. Its induced CR structure is the distribution $D=TM\cap JTM$ with
the operator of complex multiplication $J:T\C^n\to T\C^n$ restricted to $D$. 
The CR dimension is the rank $\dim_\C D=\tfrac12\dim_\R D$,
and we will focus on the simplest possible case of CR dimension~1. The CR codimension $\dim_\R M-\dim_\R D$
will be arbitrary in our work. We will assume throughout that $D$ bracket-generates $TM$ 
but the length (which is not the same as CR codimension) will be un-restricted. 
The bracket-generating (aka CR minimal) property implies that $(M,D,J)$ is Levi-nondegenerate 
(for CR-$\dim=1$), so its symmetry algebra (both infinitesimal and local) is finite-dimensional. 

A classical problem, going back to H.\ Poincar\'e \cite{Po}, is to find an effective bound on the dimension 
of local or global automorphism group of CR manifolds. This problem is by now well-understood 
for Levi-nondegenerate hypersurfaces, but is still wide open 
for Levi-degenerate higher-nondegenerate hypersurfaces and for Levi-nondegenerate CR structures of higher CR codimension.
Our first result is a sharp bound on this dimension for the class under consideration.
This was proven in a particular case of totally nondegenerate models (specific assumption on the growth vector)
in \cite{Su} (prolongation rigidity\footnote{The term ``rigid'' 
in CR context refers to a specific form of polynomial models; for vector distributions it means finite-dimensionality of their 
symmetry, while prolongation rigidity strengthens this to vanishing of the positive Tanaka prolongation.}
for more general totally nondegenerate CR structures was proven in \cite{SS,Gr}); we require no assumptions of that kind.

 \begin{theorem}\label{Th1}
The symmetry dimension of the bracket-generating CR manifolds $(M,D,J)$ of CR dimension 1 does not exceed $m+2$, 
where $m=\dim M$, provided $m>3$. In fact, for every CR symbol, the most symmetric structures of this kind 
have symmetry dimension $m+r$, where $1\leq r\leq 2$.
 \end{theorem}
 
The explicit value of $r$ depends on the underlying vector distribution $D$ and the choice of $J$ with respect to it;
we will specify in Section \ref{S6} when $r=2$ is achieved. The Tanaka prolongation procedure, used to prove Theorem \ref{Th1}
in Section \ref{S3}, leads to construction of generalized frame bundles and hence to the following statement.
Recall that strongly regular distributions, in Tanaka terminology, have the same symbol everywhere,
and in the context of CR structures, they have the same CR symbol at all points.

 \begin{cor}\label{Cor1}
For the class of strongly regular CR manifolds of CR dimension 1, there exists a Cartan connection on the frame bundle over it. 
This gives an equivariant solution to the equivalence problem. Moreover, there is a canonical linear connection on the
base manifold $M$.
 \end{cor}

The equivariancy means that the differential invariants of CR structures descend (from the frame bundle) to the original manifold, 
and hence are free of auxiliary group parameters, unavoidable with the general construction of absolute parallelism
by the method of moving frames, see partial cases \cite{BES,MS}. 
A Cartan connection gives a finer structure of curvatures, in particular allowing
to control symmetry reductions in twistor-like constructions. 
An invariant linear connection on the base allows for direct holonomy reductions.
 
Forgetting about the complex structure, $D$ is a smooth rank 2 distribution on a real manifold $M$. 
Nonholonomic (bracket-generating) distributions of rank 2 are well-studied. In particular, the most symmetric distribution 
of this kind is the contact distribution in the jet space $J^{m-2}(\R)$, also known as the Goursat distribution; it is the only one
of infinite type (possessing infinite-dimensional symmetry algebra). 
We will recall the details in Section \ref{S2}, 
in particular noticing that this distribution has reduced growth vector $(2,1,\dots,1)$,
but remark momentarily 
that the Bloom-Graham type of such structure is $(2,3,4,\dots,m)$.

In the following statement the term ``model'' is understood as a polynomial model of Beloshapka \cite{B1,B2}, 
but in our context it can be also thought of as a standard model of Tanaka \cite{T}.


 \begin{theorem}\label{Th2}
Let $m=n+1$ be the dimension of a homogeneous CR model $M\subset\C^n$ with nonholonomic CR structure $(D,J)$ 
of CR dimension 1 and reduced growth vector $(2,1,\dots,1)$. 
If $m$ is odd $>2$ or $m=4$, then the model is unique: it is a tube over the affine homogeneous curve 
 \begin{equation}\label{GOU}
\gamma=\{x_k=\lambda^k:\lambda\in\R\}_{k=1}^n\subset\R^n:\quad M=\gamma\times i\R^n.
 \end{equation}
If $m$ is even, $m\ge6$, then either $M$ is a tube as above or $M$ is non-tubular; the latter case is based on 
one distribution $D$ (different from the tubular one) with a one-parameter family of complex structures $J$ on it.
In both cases, for $m>3$, the symmetry dimension is $m+1$ (in terms of Theorem \ref{Th1}, $r=1$).
 \end{theorem}

We will exhibit the CR structures abstractly in Section \ref{S4}.
Let us discuss small dimensions. For $n=2$ the unique model structure is the standard quadric in $\C^2$ \cite{Po}, 
which is the maximal symmetry model $SU(1,2)/B$ ($B$ stands for the Borel subgroup). 
For $n=3$ the unique model structure of CR codimension 2 is sub-parabolic, meaning its distribution $D$ comes from the 
flag variety $Sp(4,\R)/B$, while the complex structure $J$ reduces it to $P'/G_0'$, where $P'=G_0'\ltimes\exp(\m)$ 
is the subgroup of opposite parabolic\footnote{Numeration of parabolic subgroup agrees with Bourbaki, 
it corresponds to crossed nodes on the Dynkin diagram.} subgroup $B^\text{op}$ to $B=P_{12}$ with the homothety
$G_0'=\R_\times\subset B^\text{op}\cap B=G_0\simeq(\R_\times)^2$, 
and $\m=\g_-$ for the Borel grading of $\g=\mathfrak{sp}(4,\R)$ corresponding to the Engel distribution in $\R^4$.
For $n=4$ the model is also unique \cite{Z}, its distribution is no longer parabolic but Goursat.
For higher dimensions the unicity is violated, for instance in the case $n=5$ in addition to the Goursat distribution
there is the rank 2 distribution associated to the flag variety $G_2/B$ and it carries a one-parameter family of CR structures.
These will be realized explicitly as real submanifolds $M^6\subset\C^5$ in Section~\ref{S7}.

Theorem \ref{Th2} implies the unicity of homogeneous tubular CR models
of CR dimension 1 and reduced growth vector $(2,1,\dots,1)$. 
The tube condition is a partial case of rigid surfaces, where $M$ is given in coordinates
by $\op{Re}(z_k)=f_k(z_1,\bar{z}_1)$, $k=2,\dots,n$. 
(Tubular CR manifolds have an Abelian subalgebra of symmetries of $\dim=n$, while rigid surfaces have that of $\dim=n-1$.) 
The following result was also announced in \cite{Pa}.

 \begin{cor}\label{Cor2}
Rigid homogeneous CR models $M\subset\C^n$ of CR dimension 1 and reduced growth vector $(2,1,\dots,1)$ 
are unique (up to local biholomorphism) in every CR codimension $n>1$.
 \end{cor}
 
Finally, we consider the general models of CR distributions of CR dimension 1. Without complex structure $J$ this
corresponds to rank 2 distributions $D$, which were studied in detail in \cite{AK,DZ}.
If the reduced growth vector starts $(2,1,1,\dots)$, including but not limited to structures considered in Theorem \ref{Th2},
the distribution $D$ allows reduction (local in $M$) by the Cauchy characteristic of $D^2$. 
If the reduced growth vector starts $(2,1,2,\dots)$, the derived distribution $D^2$ is free of Cauchy characteristics,
and then Theorem 3 of \cite{AK} claims that such flat (model) rank 2 distribution is a result of successive integrable extensions 
of the Hilbert–Cartan equation $y'=(z'')^2$ supported on the 5-manifold $\bar{M}=G_2/P_1$. We conclude:

 \begin{theorem}\label{Th3}
Let $m=n+1$ be the dimension of a homogeneous CR model $M\subset\C^n$ with nonholonomic CR structure $(D,J)$ 
of CR dimension 1 with the symbol $\m=\g_-$ of $D$ of depth $\nu\ge3$. 
Then $M=M_\nu$ allows a tower of symmetry reductions to CR manifold $(M_i,D_i,J_i)$ of CR dimension 1,
successively decreasing in depth $i$. This ends either with the Engel CR structure of reduced growth $(2,1,1)$
or with the Hilbert–Cartan CR structure of reduced growth $(2,1,2)$. Both reduce to the CR quadric in $\C^2$.
Conversely, any homogeneous CR model $(M,D,J)$ is obtained from the reduced $(\bar{M},\bar{D},\bar{J})$ 
by a sequence of so-called integrable extensions.
 \end{theorem}
 
We will explain integrable extensions, as well as prolongations, symmetry reductions and reductions by Cauchy characteristics 
in Section \ref{S5} and also exhibit some important models abstractly. 
Then we will realize the simplest of them in coordinate form in Section \ref{S7}.
Let us indicate that part of the results in this paper were known earlier (as we reference) yet in different contexts, 
and we unite all those here.

\section{Bloom-Graham type, growth vector and Tanaka prolongation}\label{S2}

Let us formulate basic definitions and recall some fundamental results of the theory.
Given a vector distribution $D$ on a manifold $M$ its strong derived flag is 
 $$
D^{[1]}=D,\ D^{[2]}=[D,D],\ D^{[3]}=[D^{[2]},D^{[2]}],\ D^{[4]}=[D^{[3]},D^{[3]}],\ \dots
 $$
and its weak derived flag is 
 $$
D^{1}=D,\ D^{2}=[D,D],\ D^{3}=[D,D^{2}],\ D^{4}=[D,D^{3}],\ \dots
 $$
Here and in what follows, by the brackets of distributions we mean those by the corresponding module of sections, and if 
the resulting $C^\infty(M)$-module is projective (which we adopt) it is the module of sections of the resulting distribution.
Thus $\Gamma([D,D'])=[\Gamma(D),\Gamma(D')]$, etc. For every point $x\in M$ these flags define filtrations of $T_xM$
(depending on $x$ but we will often omit the point from notations) $\ldots\subset D^{[k]}\subset D^{[k+1]}\subset\dots$
and $\ldots\subset D^k\subset D^{k+1}\subset\dots$; these are subordinated as follows: $D^k\subset D^{[k]}$.

By the iterated Jacobi identity, the conditions $\exists\,r:D^{[r]}=TM$ and $\exists\,s:D^s=TM$ are equivalent, even though
the minimal such $r$ may be smaller than the minimal such $s$. It turns out that the weak derived flag is more convenient
for computations with invariants of $D$. In the case $D$ is equipped with a complex structure $J$, $J^2=-\op{id}_D$,
which satisfies the integrability condition, the pair $(D,J)$ is called (an abstract) CR structure. 
The Bloom-Graham type of such structure specifies the growth of the weak derived flag of $D$, 
even including pseudo-stabilizations, when the growth stops at the given point $x\in M$ for some steps.

In the regular case, when dimensions of the flag filtrands are constant\footnote{In this case the bracket-generating
property is equivalent to CR manifold being minimal; in general, minimality is equivalent to the bracket-generating property
in an open dense subset of $M$.} (in the analytic case, the regularity is achieved on
an open dense subset of $M$) the strong growth vector is $(d_{[1]},d_{[2]},\dots)$, $d_{[k]}=\op{rank}D^{[k]}$,
while the weak growth vector is $(d_{1},d_{2},\dots)$, $d_{k}=\op{rank}D^{k}$. We will also use the
reduced growth vector $(r_1,r_2,\dots)$, where $r_k=d_k-d_{k-1}$ for $k>0$ ($d_0=0$).  

We will however require a stronger invariant associated to $D$ at $x$ (symbol of $D$ or Carnot algebra) 
 $$
\m(x)=\oplus_{i<0}\g_i=\g_{-\nu}\oplus\dots\oplus\g_{-1},
\quad\text{where }\g_{-k}=D^k/D^{k-1},\ D^0=0.
 $$
This is not only a collection of vector spaces (which can be characterized by their dimensions) but also a
graded nilpotent Lie algebra (GNLA), with bracket of vectors (or rather their equivalence classes in GNLA) 
induced by commutators of representative tangent vector fields, see \cite{T}. 

Abstractly, we consider GNLA $\m=\oplus_{i<0}\g_i$ that satisfy two fundamental properties: (i) $\dim\m<\infty$,\linebreak
(ii) $\m$ is bracket-generated by $\g_{-1}$. 
For $\dim\g_{-1}=2$ the last property implies that our distribution $D$ is free of Cauchy characteristics, 
which correspond to the center in gradation $-1$: $\mathfrak{z}(\m)\cap\g_{-1}=0$.

The Tanaka prolongation of $\m$ is the maximal graded Lie algebra $\g=\op{pr}(\m)=\oplus\g_k$ such that $\m=\g_-$
and $[\g_{-1},x]=0$ for $x\in\g_{\ge0}$ implies $x=0$. This is constructed successively in grading $k$.
In particular, $\g_0=\mathfrak{der}_0(\m)$ consists of grading preserving derivations (it always contains the
grading element $Z$ given by $[Z,x]=kx$ for $x\in\g_k$) and $\g_1$ is identified with the space of elements in 
$\op{Hom}(\g_{-1},\g_0)$ that (uniquely) extend to a homomorphism $\m\to\m\oplus\g_0$ of degree $1$ 
satisfying the Jacobi identity, whenever defined. 

For CR structures, the 0-th prolongation\footnote{In order to avoid confusion we will denote, in what follows, 
by $\hat{\g}_0=\mathfrak{der}_0(\m)$ the 0-th prolongation for the distribution $D$.} consists of derivations of $\m$ of degree 0
that preserve $J$ on $D$: 
 \begin{equation}\label{eqg0}
\g_0=\mathfrak{der}_0(\m)\cap\mathfrak{gl}(\g_{-1},J)\subset\op{End}(\g_{-1}). 
 \end{equation}
The prolongation $\g_1$ is defined as the subspace of vectors in $\op{pr}_1(\m)$ for which the bracket with $\g_{-1}$ lands 
in $\g_0$, given by \eqref{eqg0}; the higher prolongations $\g_k$ are modified similarly. 
This gives the prolongation $\g=\op{pr}(\m,\g_0)$. 

If these Tanaka algebras $\g(x)$ are finite-dimensional for all $x\in M$, then the symmetry algebra $\mathfrak{s}$ of $D$, 
or respectively $(D,J)$, is bounded in dimension as follows:
 \begin{equation}\label{dims}
\dim\mathfrak{s}\leq\mathfrak{inf}_{x\in M}\dim\g(x).
 \end{equation}
For strongly regular structures, when GNLA $\g(x)$ are isomorphic at all points $x$, this is the result of \cite{T}.
In the general case it was proven in \cite{K,K2}. This gives a universal bound to the symmetry dimension.

We will focus on rank 2 distributions $D$, i.e.\ $\dim\g_{-1}=2$. An important class of distributions
consist of those with reduced growth vector $(2,1,1,\dots,1)$. This contains the Goursat distributions \cite{C},
that are the Cartan distributions on the jet spaces $J^k(\R)$. Recall that by the Lie-B\"acklund theorem \cite{AK} the symmetries of 
all such distributions are prolongations of the contact algebra from $J^1(\R)$, so in this case $\mathfrak{s}=\mathfrak{cont}(3)$. 
In fact, these are the only distributions of rank 2 with infinite-dimensional symmetry algebra \cite{C,MT}.

By the von Weber theorem \cite{We,GKR} the distribution is Goursat if and only if its strong and weak growth vectors 
coincide everywhere $d_{[k]}=d_k=k+1$ for $k>0$. The condition $d_k=k+1$ is not sufficient for the equivalence to Goursat,
and we will describe all other possibilities in Section \ref{S4}. But the condition $d_{[k]}=k+1$ is generically sufficient,
meaning the isomorphism of germs at generic points $x\in M$. This does not extend to all points, and the stratification
of non-regular points form a monster tower \cite{MZ}. The corresponding Bloom-Graham type is
a longer vector, and the model is more complicated. However we will restrict to regular points and such singularities will not occur.
(See, however, Remark \ref{Rk1} in Section \ref{S5}.)

Let us also make a remark on distributions on generalized flag varieties $M=G/P$, where $G$ is semi-simple and $P$ parabolic.
The Lie algebra $\mathfrak{p}=\op{Lie}(P)\subset\g=\op{Lie}(G)$ is marked by crosses on the Dynkin diagram;
up to conjugation its choice is equivalent to $\mathbb{Z}$-gradation 
$\g=\g_{-\nu}\oplus\dots\oplus\g_0\oplus\dots\oplus\g_\nu$,
with $\m=\g_{-\nu}\oplus\dots\oplus\g_{-1}\simeq T_oM$ and $\mathfrak{p}=\g_0\oplus\dots\oplus\g_\nu$
(simple roots not in $\g_0$ correspond to crossed nodes). 
The reductive part $\g_0$ has Levi part $\g_0^{ss}$ corresponding to removal of crosses on the Dynkin diagram and 
the center $\mathfrak{z}(\g_0)$ has rank equal to the number of crosses. The (real) distribution $D$ on $\exp\m$ corresponds
to $\g_{-1}$ and is nonholonomic of rank 2 only for the cases $A_2/P_{12}$, $B_2/P_2=C_2/P_1$,
$B_2/P_{12}=C_2/P_{12}$, $G_2/P_1$ and $G_2/P_{12}$. In the first two cases $\m=\mathfrak{heis}(3)$ and so $(M,D)$
is a contact 3-manifold; in the third case $\m$ has reduced growth vector $(2,1,1)$ and $(M,D)$ is an Engel structure.
In all these cases $\op{pr}(\m)=\mathfrak{cont}(3)$.
On the other hand, for the last two indicated cases when $\g=\op{Lie}(G_2)$ we have $\op{pr}(\m)=\g$.

\section{Prolongation rigidity and symmetry bound}\label{S3}

Let $\m=\g_{-\nu}\oplus\dots\oplus\g_{-1}$ be the symbol of the distribution $D$, or abstractly a fundamental GNLA.
The number $\nu$ is called the depth of $\m$. For $\dim\g_{-1}=2$, we get $\g_0\subset\op{End}(\g_{-1})=\mathfrak{gl}(2,\R)$.
However $\g_0$ must preserve a complex structure $J$ on $\g_{-1}$, so $\g_0\subset\mathfrak{gl}(1,\C)=\mathfrak{co}(2)$.
Thus $\dim\g_0=r\in\{1,2\}$, and if $r=2$ then $\g_0=\mathfrak{gl}(1,\C)=\langle Z,J\rangle$, 
while if $r=1$ then $\g_0=\R(Z)$ is generated by the grading element $Z$.

By the non-holonomic assumption, $\nu>1$. If $\nu=2$, then $\m=\mathfrak{heis}(3)$ and 
$\g_0=\mathfrak{gl}(1,\C)$, so we have $\op{pr}(\m,\g_0)=\mathfrak{su}(1,2)$,
which corresponds to the classical case of 3-dimensional CR sphere.

 \begin{prop}\label{P1}
If $\nu>2$ then $\op{pr}_+(\m,\g_0)=0$. 
 \end{prop}
 
 \begin{proof}
Assume, on the contrary, that $\g_1$ is nontrivial. From $\g_1\subset\op{Hom}(\g_{-1},\g_0)$ it has $\dim\leq4$, 
but using prolongation and the iterated Jacobi identities one may conclude that its dimension is bounded by 2. 
We however invoke Lie theory to get the requested stronger result.

We first claim that $(\m,\g_0)$ has finite type. As already noted, the only infinite type $\m$ with $\dim\g_{-1}=2$
corresponds to the Goursat distribution with $\op{pr}(\m)=\mathfrak{cont}(3)$. If $\nu=2$ then 
$\op{pr}(\m,\g_0)=\mathfrak{su}(1,2)$, while if $\nu>2$ then $\g_{-1}$ has distinguished line, so $\g_0=\R(Z)$
and it is well known that $\op{pr}_+(\m,\g_0)=0$ in this case.
 
Following \cite[Proposition 2.4]{Y} we decompose $\g_{-1}$ into $\g_0$ irreps 
and note that $\g_1$ consists 
of dual copies of some of them. Thus $\dim\g_1\leq2$, and in the case of equality we have $\g_1=\g_{-1}^*$.
Moreover, by Proposition 2.5 of \cite{Y} the semi-simple part $\g^{ss}$ of $\g=\op{pr}(\m,\g_0)$ contains the entire positive part.

Thus if $r=\dim\g_0$ then $\g^{ss}$ is generated by $\g_{-1}'\oplus\g_0'\oplus\g_1$ and has rank $\leq r$;
here $\g_{-1}'\subset\g_{-1}$ is dual to $\g_1$ and $\g_0'=[\g_{-1}',\g_1]$.
Consider at first the case $\op{rank}(\g^{ss})=2$, so $r=2$. Since $\g_0$ is Abelian, the grading of $\g^{ss}$, inherited from $\g$,
is abstractly induced by the Borel subalgebra $\mathfrak{b}=\mathfrak{p}_{12}$.
The algebras $A_1\oplus A_1=\mathfrak{so}(2,2)$, $B_2=\mathfrak{so}(2,3)\simeq\mathfrak{sp}(4,\R)$ and
$G_2$ have no $\g_0$-invariant complex structure on $\g_{-1}$. From the remaining case $A_2$ the real form 
$\mathfrak{sl}(3,\R)$ lacks the same property, but the real form $\mathfrak{su}(1,2)$ 
(admitting parabolic subalgebra $\mathfrak{p}_{12}$) has gradation $\g_{-2}\oplus\g_{-1}\oplus\g_0\oplus\g_1\oplus\g_2$ 
with dimensions $(1,2,2,2,1)$ and $\g_0=\mathfrak{gl}(1,\C)$. Since $\g_{-1}$ bracket-generates $\m$, we conclude 
$\m=\g_{-2}\oplus\g_{-1}=\mathfrak{heis}(3)$. Thus this case corresponds to a CR manifold of dimension 3, 
so the CR structure is of hypersurface type and has depth $\nu=2$.
 
Next consider the case $r=2$ but $\op{rank}(\g^{ss})=1$. This implies $\dim\g_1=1$, so $\g^{ss}=\mathfrak{sl}(2,\R)$.
Due to structure relations of $\mathfrak{sl}(2,\R)$ we have $\g_0\cap\g^{ss}=\langle Z\rangle$.
Thus $\g_+=\g_1$, $\g_0/\langle Z\rangle\simeq\R(J)$ and the nilradical of $\g$ is $\m\ominus(\g_{-1}\cap\g^{ss})$.
This radical decomposes into irreps of $\mathfrak{sl}(2,\R)$, but on each of them $Z$ must have a symmetric spectrum.
Yet $Z$ acts with negative eigenvalues on $\m$. Thus this case cannot happen.

Finally, consider the case $r=1=\op{rank}(\g^{ss})$, in which case $\dim\g_1=1$. Here again $\g^{ss}=\mathfrak{sl}(2,\R)$,
the radical is $\g\ominus\g^{ss}=\g_{-\nu}\oplus\dots\oplus\g_{-2}\oplus(\g_{-1}\ominus(\g_{-1}\cap\g^{ss}))$, 
and the same argument finishes the proof.
 \end{proof}

A similar though not identical proof was given in \cite[\S5.6]{MN}. 
(The author got notice of this fact after posting in arXiv; our proof remains here for completeness and
proper references not noted in \cite{MN}.)
 
 \begin{proof}[Proof of Theorem \ref{Th1}]
This follows from Proposition \ref{P1} and estimate \eqref{dims}. 
Note that we do not require homogeneity or strong regularity of $(M,D,J)$ for the claim.
Equality is attained on the standard model $M=\exp(\m)$,
$D=\g_{-1}$ equipped with the left-invariant complex structure $J$ read off $\g_{-1}$.
The model possesses $r$ more symmetries, so locally $M$ coincides with $G/H$, where $\op{Lie}(G)=\g=\op{pr}(\m,\g_0)$
and $H=\exp_G(\g_0)\subset G$.
 \end{proof}
 
Let us recall the standard Tanaka geometric prolongation \cite{T} of a nonholonomic geometric structure,
in our case $(D,J)$. It is a sequence of fiber bundles $\mathcal{F}_j$ (over $\mathcal{F}_{j-1}$ and over $M$), $j\ge0$, 
with the induced filtered tangent structure, stabilizing at a frame bundle $\mathcal{F}\to M$ 
with an absolute parallelism $\omega\in\Omega^1(\mathcal{F},\g)$, 
which for every $u\in\mathcal{F}$ gives an isomorphism $\omega_u:T_u\mathcal{F}\to\g$.
Referring to \cite{Ze,KST} for details, the basic step of construction is a choice of partial connections 
$\mathcal{H}^i_j$, $i\leq j$, yielding the structure functions $c^\pm_{\mathcal{H}^i_j}$.

In the prolongation rigid case $\op{pr}_+(\g,\g_0)=0$ we have $c^+_{\mathcal{H}^i_j}=0$, see \cite[\S3.5]{KST}.
Under a change of partial connections, the structure function $c^-_{\mathcal{H}^i_j}$ is modified by the Spencer operator 
$\delta$ in the generalized Spencer complex. In our case it is the Chevalley-Eilenberg bi-graded complex 
$\g\otimes\Lambda^\bullet\m^*$ for the Lie algebra cohomology of $\m$ with coefficients in $\g$:
 \begin{equation}\label{Spencer}
0\to\g\to\g\otimes\m^*\stackrel{\delta}\longrightarrow\g\otimes\Lambda^2\m^*\to\dots
 \end{equation}
where $(\g\otimes\m^*)_\ell=\oplus_{i>0}\g_{\ell-i}\otimes(\m_{-i})^*$, 
$(\g\otimes\Lambda^2\m^*)_\ell=\oplus_{i,j>0}\g_{\ell-i-j}\otimes(\m_{-i})^*\wedge(\m_{-j})^*$, etc.
Normalizations are sequences of complements $N^-_\ell\subset(\g\otimes\Lambda^2\m^*)_\ell$ to the image of $\delta$,
allowing to successively construct the bundles $\mathcal{F}_{j+1}\to\mathcal{F}_j$ by the condition 
$c^-_{\mathcal{H}^i_j}\in N^-_{j+1}$, which restricts the choice of partial connections $\mathcal{H}^i_j$.

If the normalization conditions can be chosen invariantly with respect to the action of the group 
$H=G_0\ltimes\exp\bigl(\op{pr}_+(\g,\g_0)\bigr)$, then $\mathcal{F}\to M$ becomes a principal $H$-bundle 
and $\omega$ becomes $H$-equivariant, thus yielding a Cartan connection. In our case $H=G_0$ is reductive,
so an invariant complement exists.
 
 \begin{proof}[Proof of Corollary \ref{Cor1}]
The case $\nu=2$ corresponds to CR-structures of CR-dimension and CR-codimension 1,
in which case the frame bundle is modeled on $SU(1,2)/B$ and the construction of the connection goes back to Cartan \cite{C}.
(In this case $\g$ is simple and, referring to the Tanaka construction, the choice of a complement is given by the kernel 
of the Kostant codifferential on the Chevalley-Eilenberg complex.)
 
For $\nu>2$, by Proposition \ref{P1}, the prolongation is done in one step, 
and the resulting bundle has an Abelian structure group $H$ of rank $r=1\vee2$ acting from the right. 
It acts semi-simply on $\g=\m\oplus\g_0$ and on the associated Spencer complex. 
Hence the image of the Spencer differential $\delta$ of \eqref{Spencer} allows an $H$-invariant complement.
Every such choice yields a canonical Cartan connection. Canonical means independent of particular CR geometry
with the given symbol, but does not mean unique. Indeed, already for the simplest case $n=\nu=3$ (Engel structure) 
we get $\g_0=\R$ and this approach gives plenty normalizations, because any graded complement to 
$\op{Im}(\delta)$ is $G_0$-invariant.

A more restrictive choice of a complement comes from the construction of Morimoto \cite{Mo}, based on a positive
definite bilinear form $(\cdot,\cdot)$ on $\g$ and mappings $\tau:\mathfrak{h}\to\g$ for $\mathfrak{h}=\op{Lie}(H)$, 
$\tau_0:G_0\to G_0$ such that
 \begin{enumerate}
\item $(\g_i,\g_j)=0$ for $i\neq j$,
\item $\tau(\g_i)\subset\g_{-i}$ for $i\geq0$ and $([h,x],y)=(x,[\tau(h),y])$ for $x,y\in\g$, $h\in\mathfrak{h}$,
\item $(ax,y)=(x,\tau_0(a)y)$ for $x,y\in\g$, $a\in G_0$.
 \end{enumerate}
For any such metric the normalization $N^-=\op{Im}(\delta)_+^\perp$ is $H$-invariant.

In our case, $\mathfrak{h}=\g_0$ is either $\mathfrak{gl}(1,\R)$ or $\mathfrak{gl}(1,\C)$.
The space $\g_{-1}$ has a unique up to scale $J$-invariant metric, which induces the metric on the
free Lie algebra $\mathfrak{u}$ generated by $\g_{-1}$. Since $\m$ is a quotient of $\mathfrak{u}$ by an ideal,
it also gets an induced metric. We complement this (direct sum $\g=\m\oplus\g_0$) by a metric on $\h$ 
and define the map $\tau_0$ as transposition and $\tau$ as its differential (identity for the grading element and $-J$ for $J$). 
These satisfy all conditions (1-3) and thus define a normal Cartan connection $\omega\in\Omega^1(\mathcal{F};\g)$. 

The fundamental invariant of this connection is its curvature $K=d\omega+[\omega,\omega]$. It is horizontal,
meaning that hooking a vertical vector field as one of the arguments yields zero. The components of $K$ with respect to horizontal
vectors of the frame together with their iterated derivatives along the horizontal vectors of the frame give all
differential invariants of the CR structure $(M,D,J)$, thus solving the equivalence problem.

Finally, the principal bundle $\mathcal{F}\to M$ has reductive structure group $\R_\times$ or $CO(2)$ and so it splits the affine
connection into the principal connection and the linear connection on the base.
 \end{proof}

\section{Growth vector $(2,1,\dots,1)$}\label{S4}

In this section we describe GNLA $\m=\g_{-n}\oplus\dots\oplus\g_{-1}$ with 
$r_1=\dim\g_{-1}=2$, $r_i=\dim\g_{-i}=1$ for $i>1$.
We will denote the basis of $\m$ so: $\g_{-1}=\langle e_1',e_1''\rangle$, $\g_{-i}=\langle e_i\rangle$ for $i>1$.
Here $\dim\m=m=n+1$.

Recall that the Goursat distribution is the Cartan distribution of the jet-space $J^{n-1}(\R)$ given in jet-coordinates 
$(x,y_0,y_1,\dots,y_{n-1})$ by $D=\langle \p_x+y_1\p_{y_0}+y_2\p_{y_1}+\dots+y_{n-1}\p_{y_{n-2}},\p_{y_{n-1}}\rangle$.
Its GNLA $\m$ is given by the following relations (as is customary, we indicate only non-trivial relations):
 \begin{equation}\label{Gou}
[e_1',e_1'']=e_2,\ [e_1',e_2]=e_3,\ [e_1',e_3]=e_4,\ \dots,\ [e_1',e_{n-1}]=e_n. 
 \end{equation}
 
 \begin{prop}\label{P2}
A graded nilpotent Lie algebra with reduced growth vector $(2,1,1,\dots)$ is either Goursat given by \eqref{Gou}
or is given by the following relations for $n=\dim\m-1=2k+1$ $(k\in\Z,k>1)$:
 \begin{multline}\label{nGou}
[e_1',e_1'']=e_2,\ [e_1',e_2]=e_3,\ \dots,\ [e_1',e_{n-2}]=e_{n-1},\\
[e_1'',e_{n-1}]=-[e_2,e_{n-2}]=[e_3,e_{n-3}]=...=(-1)^{k-1}[e_k,e_{k+1}]=e_n.  
 \end{multline}
 \end{prop}
 
 \begin{proof}
Given a GNLA $\m$ of depth $n$, the quotient by a central subspace $\g_{-n}$ is a GNLA of depth $n-1$.
Any GNLA with the required growth can be obtained in this way, so we proceed successively in $n$.
 
It is easy to see that for $n<5$ there is only one GNLA with the required properties, it is given by relations \eqref{Gou}.
To extend this to $n=5$, take the bracket of $[e_1'',e_3]=0$ with $e_1'$ and get the following relations:
 $$
[e_1',e_4]=a\,e_5,\quad [e_1'',e_4]=-[e_2,e_3]=b\,e_5 
 $$
for some $(a,b)\in\R^2\setminus0$. If $b=0$, then rescale $e_5$ to get $(a,b)=(1,0)$ so the Goursat type \eqref{Gou}$_{n=5}$.
If $b\neq0$ then modify $e_1'$ by $e_1''$ and rescale $e_5$ to get $(a,b)=(0,1)$ so \eqref{nGou}$_{n=5}$.
This corresponds to the grading of $\g=G_2$ with respect to Borel parabolic subalgebra $\mathfrak{p}_{12}$.

Next, consider $n=6$. If $\m$ is an extension of \eqref{Gou}$_{n=5}$, take the bracket of $[e_1'',e_4]=0$ with $e_1'$ 
and get the following relation (the second equality is obtained by bracketing $[e_2,e_3]=0$ with $e_1'$):
 $$
[e_1'',e_5]=-[e_2,e_4]=[e_3,e_3]=0.
 $$
This must be accompanied with $[e_1',e_4]=a\,e_5$ for $a\neq0$, which can be renormalized to $a=1$. This gives 
the GNLA \eqref{Gou}$_{n=6}$.

If $\m$ is an extension of \eqref{nGou}$_{n=5}$, take the bracket of $[e_1'',e_4]=e_5=-[e_2,e_3]$ with $e_1'$ 
and get the following relation:
 $$
[e_1',e_5]=[e_2,e_4]=-[e_2,e_4]=0.
 $$
Similarly, taking the bracket with $e_1''$ we get $[e_1'',e_5]=0$. This $e_5$ is central but not in the last grade $\g_{-6}$,
and $e_6$ is not bracket-generated. This is a contradiction with fundamental property, so cannot happen.

The rest is by induction, and the induction step resembles the steps above for $n=5,6$. Indeed, assuming $n=2k+1$ odd,
the GNLA $\m$ of the previous depth $n-1=2k$ has only one type, and we can extend it by the following relation:
 $$
[e_1',e_{2k}]=a\,e_n,\quad [e_1'',e_{2k}]=-[e_2,e_{2k-1}]=\dots=(-1)^{k-1}[e_k,e_{k+1}]=b\,e_n
 $$
for some $(a,b)\in\R^2\setminus0$. The cases $b=0$ and $b\neq0$ leads to the types \eqref{Gou} and \eqref{nGou}
respectively. On the other hand, if $n=2k+2$ is even, then the GNLA $\m$ of the previous depth $n-1=2k+1$ 
can have two types: type \eqref{Gou} has a unique extension by the Jacobi identity, while type \eqref{nGou} 
is non-extendable with the fundamental property.
 \end{proof}
 
Proposition \ref{P2} is originally due to \cite{V}, as was indicated by a reviewer. The proofs are similar; we keep ours for
completeness. 
 
 \begin{prop}\label{P3}
For the algebraic models of Proposition \ref{P2} the GNLA \eqref{Gou} admits only one complex structure $J$ on $\g_{-1}$
up to conjugation, while the GNLA \eqref{nGou} admits a one-parametric family of such structures.
Namely, in the given bases $Je_1'=e_1''+be_1'$, where $b=0$ for \eqref{Gou} while $b$ is arbitrary for \eqref{nGou}.
 \end{prop}
 
 \begin{proof}
The complex structure for \eqref{Gou} in the indicated basis is uniquely given by the relation $Je_1'=ae_1''+be_1'$,
where $a\neq0$. The algebra $\hat\g_0=\mathfrak{der}_0(\m)$ that is the 0-prolongation of the symbol of the distribution
is the Borel subalgebra $\mathfrak{b}\subset\mathfrak{gl}(2,\R)$. The corresponding group $B$ can transform the parameters $(a,b)\mapsto(1,0)$. 

In case \eqref{nGou} we have the same formula for $J$ via $a,b$, but now $\hat\g_0=\mathfrak{der}_0(\m)$ 
is two-dimensional Abelian $\R\oplus\R$: Cartan subalgebra $\mathfrak{h}\subset\mathfrak{gl}(2,\R)$.
The corresponding group rescales along the distinguished axes $e_1',e_1''$, and we can eliminate one parameter
$(a,b)\mapsto(1,b)$. However then the parameter $b$ is non-removable.
 \end{proof}

Recall that graded Lie algebra $\g$ with fundamental $\m=\g_-$ has filtration $\g^i=\oplus_{j\geq i}\g_j$.
A filtered Lie algebra $\mathfrak{f}$ is called a filtered deformation of graded Lie algebra $\g$ if $\op{gr}(\mathfrak{f})=\g$
as graded Lie algebras. A graded subalgebra $\g$ is filtration-rigid if any filtered deformation $\mathfrak{f}$ of $\g$ 
has $\mathfrak{f}\simeq\g$ as filtered Lie algebras.

 \begin{lemma}\label{FiltrRigid}
If a graded Lie algebra $\g$ contains a grading element $Z$, $\op{ad}_Z|_{\g_i}=i\cdot{\bf1}$, then it is filtration rigid.
 \end{lemma}
 
 \begin{proof}
The result is well-known, cf.\ \cite{CK}. The argument is as follows: if $\tilde{Z}\in\mathfrak{f}$ corresponds to $Z$,
it is diagonalizable and thus defines a gradation of $\mathfrak{f}$ that yields an isomorphism to $\g$.
 \end{proof}
 
It is known \cite[Theorem 3]{K} that a symmetry of a finite type nonholonomic structure on $M$ is a filtered deformation
of a graded subalgebra of $\g=\op{pr}(\m,\g_0)$.

 \begin{cor}\label{Tanaka-flat}
If a strongly regular nonholonomic structure of type $(\m,\g_0)$, with $\g_0$ containing the grading element $Z$,
has symmetry algebra of dimension $\dim\g$,  where $\g=\op{pr}(\m,\g_0)$, then the structure is Tanaka-flat:
$M$ is locally isomorphic to $\exp\m$ with left-invariant structure on the distribution $D$ corresponding to $\g_{-1}$.
 \end{cor}

Indeed, in this case the symmetry is a filtered deformation of $\g$, which is filtration rigid.

 \begin{proof}[Proof of Theorem \ref{Th2}]
A model of CR-structure is a coordinate representation of a surface in the complex space that is homogeneous 
with respect to some weights of coordinates. This is equivalent to the existence of a scaling field 
$X=\sum\lambda_kz_k\p_{z_k}$ that is a symmetry of the surface. 
Since $X$ belongs to the stabilizer of a point $o\in M$, and $\g_0=\op{Lie}(H)$ contains only one real semi-simple element $Z$,
the weights $\lambda_k\in\mathbb{N}$ are proportional to the eigenvalues of $Z$.
Together with homogeneity and fixed Bloom-Graham type $(2,3,\dots,n+1)$ this means the projective tuple
$[\lambda_1:\dots:\lambda_n]$ is equivalent to $[1:2:\dots:n]$ that gives weight $k$ to $z_k$ ($1\leq k\leq n$).

In algebraic language, a model corresponds to $\m$ with grading element $Z$ included into $\g_0$. If it is homogeneous, 
then $\m(x)$ are isomorphic at all points and the model may be identified with $M=\exp(\m)$ by Corollary \ref{Tanaka-flat}.
Thus, classification of homogeneous models is reduced to classification of fundamental GNLA with a complex structure $J$ 
on $\g_{-1}$, so the result follows from Propositions \ref{P2} and \ref{P3}.
 \end{proof}
 
 \begin{proof}[Proof of Corollary \ref{Cor2}]
Rigid real submanifolds $\op{Re}(z_k)=f_k(z_1,\bar{z}_1)$, $k=2,\dots,n$, of CR dimension 1 in $\C^n$ have an Abelian 
subalgebra of symmetries $\langle\p_{\op{Im}(z_k)}:1<k\leq n\rangle$ of dimension $n-1$. 
However the maximal Abelian subalgebra
of GNLA \eqref{nGou} together with its extension by the grading element $Z$ has dimension $k+1=m/2=[(n+1)/2]$.
This is less than $n-1$ for $n\ge4$, so the alternative models \eqref{nGou} are not rigid. 
However for $n=2,3$ the models are unique by Propositions \ref{P2} and \ref{P3}. 
Thus only the standard models \eqref{Gou} remain, which are indeed rigid.
 \end{proof}


\section{General rank 2 distributions: extensions, reductions and prolongations}\label{S5}

Let us start by recalling \cite{AK} the procedure of extending GNLAs. 
Note on terminology: extension of a GNLA of depth $\nu$ is a GNLA of larger depth. 
Thus an extension is increasing the algebra in the negative direction, while prolongation is 
doing that in the positive direction. Let us consider one step extension:
 \begin{equation}\label{mmm}
\g_{-\nu-1}\oplus\underbrace{\g_{-\nu}\oplus\dots\oplus\g_{-1}}_{\m}=\tilde{\m}.
 \end{equation}
Equivalently we represent this as an exact sequence
 \begin{equation}\label{three}
0\to\g_{-\nu-1}\to\tilde{\m}\to\m\to0,
 \end{equation}
which tells that $\tilde{\m}$ is a central extension of $\m$ by $\g_{-\nu-1}$. Such extensions are encoded
by the homogeneous component $H^2(\m)_{\nu+1}$ of the second Chevalley-Eilenberg cohomology of Lie algebra $\m$ 
with trivial coefficients.
On the other hand, the graded component of the bracket decomposes as follows. Notice that the grading of the dual space is
negative of that for the original component, thus $\m_{-1}^*:=(\m_{-1})^*=(\m^*)_1$ etc. We get
 $$
(\Lambda^2\m^*)_{\nu+1}=(\m^*_{-1}\wedge\m^*_{-\nu})\oplus(\m^*_{-2}\wedge\m^*_{-\nu+1})\oplus\dots
 $$
Since $\m_{-1}$ bracket-generates $\m$, the left-hand side is fully determined by the first summand 
$\m^*_{-1}\wedge\m^*_{-\nu}$. Thus, since $H^2=Z^2/B^2$ and $\m$-coboundaries $B^2$ are trivial in grading $\nu+1$,
the cohomology is given by cocycles $Z^2$ and we conclude
 $$
H^2(\m)_{\nu+1}=Z^2(\m)_{\nu+1}\subset\m^*_{-1}\wedge\m^*_{-\nu}.
 $$
The group $G_0=\op{Aut}_0(\m)$ of grading preserving automorphisms of $\m$ (with $\op{Lie}(G_0)=\g_0$) acts 
on $H^2(\m)_{\nu+1}$ and extensions of dimension $d=d_{\nu+1}$ are bijective with $G_0$-orbits of 
$\op{Gr}_dH^2(\m)_{\nu+1}$.

For instance, for the unique 5-dimensional GNLA $\m_{HC}=\g_{-3}\oplus\g_{-2}\oplus\g_{-1}$ of depth 3 with two generators,
$H^2(\m_{HC})_3$ is $G_0$-equivariantly isomorphic to the Minkowski conformal 3-space 
$\R^{1,2}=\op{ad}(\mathfrak{sl}_2)\simeq S^2\R^2$ whence three $G_0$-orbits in $\mathbb{P}H^2(\m_{HC})_3$. 
Following \cite{AK} these are called hyperbolic $\mathfrak{h}_6$, elliptic $\mathfrak{ell}_6$ and parabolic $\mathfrak{p}_6$. 
In this way we get extensions of $\m_{HC}$ with $d=1,2,3$ and vector growth: 
three cases $(2,1,2,1)$, three cases $(2,1,2,2)$ and unique $(2,1,2,3)$.
Then one makes next extension etc.

Note that $\g_{-1}$ is unchanged and given a complex structure $J$ on it, the extension inherits it. This algebraic extension
has a geometric counter-part, called integrable extension in \cite{AK}. From CR triple $(M,D,J)$
one can construct $(\tilde{M},\tilde{D},\tilde{J})$ with the corresponding GNLAs $\m,\tilde{\m}$: consider a principal bundle
$\tilde{M}\to M$ with Abelian fiber $H=\exp\g_{-\nu-1}$, choose a principal connection and lift the CR structure $(D,J)$ 
to $(\tilde{D},\tilde{J})$ using this connection. Performed in a symmetric manner, this inherits a homogeneous structure.

The opposite operation for an extension is a symmetry reduction. On the algebraic level, it is given by the projection 
$\tilde{\m}\to\m$ of the sequence \eqref{three}. For the corresponding nilpotent Lie groups $\tilde{M}\to M$
is the submersion of the distributions preserving the complex structures $J$ and it is given by the symmetry 
reduction along $\mathfrak{h}=\g_{-\nu-1}$. More generally, on geometric level, 
if $\mathfrak{h}$ is an Abelian algebra of symmetries transversal to $\tilde{D}^{\nu}$ then on the (local) leaf space $M$
there is an induced CR structure $(D,J)$, of which $(\tilde{D},\tilde{J})$ is an integrable extension. We summarize this as follows:
 $$
(\tilde{M},\tilde{D},\tilde{J})\autoleftrightharpoons{extension}{reduction}(M,D,J)
 $$
This correspondence establishes Theorem \ref{Th3}.

\smallskip

Next we discuss two different prolongations. One is algebraic prolongation of Tanaka: $\g=\op{pr}(\m)$, which in the
case of CR structure is specified to $\g=\op{pr}(\m,\g_0)$ with $\g_0=\mathfrak{der}_0(\m)\cap\mathfrak{gl}(\g_{-1},J)$.
By the prolongation rigidity, for CR structures of CR dimension 1, $\op{pr}(\m,\g_0)=\m\rtimes\g_0$.
Related to this notion is the geometric prolongation \cite{T} of nonholonomic geometric structures, associating 
a frame bundle $\mathcal{F}\to M$ encoding the structure. Again, for our CR structures, $\mathcal{F}$ is a $G_0$
bundle over $M$ with a Cartan connection, namely (we denote by $r_g$ the right action of the structure group and by
$\zeta_v$ the corresponding fundamental field)
 \begin{equation*}\label{Cartanconnection}
\begin{tikzpicture}
\node at (0.3,1.0) {$\mathcal{F}$}; \node at (0.3,0) {$M$}; \node at (0,0.5) {\small{$G_0$}}; 
\draw[->] (0.3,0.8) -- (0.3,0.2); \node at (3,0.6) {$\omega:T\mathcal{F}\stackrel{\sim}\to\mathcal{F}\times\g$}; 
\node at (7.3,1) {$r_g^*\omega=\op{Ad}_g^{-1}\omega\ \forall g\in G_0$}; 
\node at (6.9,0.2) {$\omega(\zeta_v)=v\ \forall v\in\g_0$}; 
 \end{tikzpicture}
 \end{equation*}
 
Another notion is the prolongation of a rank 2 distribution $D$ on $M$ (this is related to prolongation
of differential equations encoding $D$, see \cite{AK} and Section \ref{S7}). 
Given a rank 2 distribution $D$ on a manifold $M$ one produces another 
rank 2 distribution on a manifold of one dimension larger. Namely, $\hat{M}=\mathbb{P}D$ with the natural projection 
$\pi:\hat{M}\to M$, and for $\ell\in\mathbb{P}D(x)$, $\pi(\ell)=x\in M$, 
we get the rank 2 distribution $\hat{D}=d\pi^{-1}(\ell)$.
For example, if we start with the holonomic rank 2 distribution $T\R^2$ on $\R^2$, then its first prolongation is 
the contact 2-distribution on $J^1(\R)\subset \mathbb{P}T\R^2$, then the Engel distribution on $J^2(\R)$, etc.

The construction is functorial and $\op{Sym}(\hat{M},\hat{D})=\op{Sym}(M,D)$. The inverse construction, called
de-prolongation, is as follows. The square $\hat{D}^2$ is the rank 3 distribution $d\pi^{-1}(D(x))$ 
with Cauchy characteristic (tangent symmetry)
$d\pi^{-1}(0)\subset\hat{D}$. The quotient by it maps $(\hat{M},\hat{D}^2)$ to $(M,D)$.
We summarize this as follows:
 $$
(\tilde{M},\tilde{D})\autoleftrightharpoons{prolongation}{deprolongation}(M,D)
 $$

Due to the existence of a Cauchy characteristic of $\hat{D}^2$, the growth vector of the prolonged distribution $\hat{D}$ is
$(2,1,1,\dots)$ at regular points and this indicates a possibility of (local) de-prolongation.
Thus, every (regular) rank 2 distribution, which allows maximal number of de-prolongations, 
is Goursat and thus the prolongation of the tangent distribution $T\R^2$. 
In the opposite case, if $\Delta$ with GNLA $\m$ is non-equivalent to the Goursat distribution then its 
(successive) de-prolongation $\bar\Delta$ with GNLA $\bar\m$ has growth vector $(2,1,2,\dots)$, so 
$\bar\m$ is an extension of $\m_{HC}$. Thus a regular rank 2 distribution $\Delta$ is either Goursat or is
a prolongation of an integrable extension of distribution $\bar\Delta$ with growth vector $(2,1,2)$.

 \begin{rk}\label{Rk1}
By the very definition of the prolongation of $D$ if $\hat{\g}_0=\op{pr}_0(\m)$ does not act transitively on $\mathbb{P}D$
then $\hat{M}$ is not strongly regular. For instance, if $\hat{\g}_0$ stabilizes $\ell\subset D(x)$ then the corresponding point
$\hat{x}=(x,\ell)$ is singular in $\hat{M}$ in the sense that the GNLA $\hat\m$ changes its type near $\hat{x}$.
In particular, in such cases, the symmetry algebra, even if acting transitively on $M$, does not allow a transitive lift to $\hat{M}$.

For a CR structure of CR-$\dim=1$ the structure algebra $\g_0=\hat{\g}_0\cap\mathfrak{gl}(\g_{-1},J)$ acts transitively 
on $\mathbb{P}D$ iff $r=2$ in the terminology of Theorem \ref{Th1}. Such structures will be described in the next section.
 \end{rk}

One should note that if $\hat{D}$ is the prolongation of a rank 2 distribution $D$ then it does not carry a natural complex
structure $\hat{J}$ even if we start with CR structure $(D,J)$. (In fact, $J$ induces a $\Z_2$ action on the fiber 
$\R\mathbb{P}^1\simeq\mathbb{S}^1$ of $\pi:\tilde{M}\to M$ if $\mathbb{P}\Delta$ is understood as line projectivization; 
however this recovers $J$ up to sign; if instead we consider ray projectivization $\mathbb{P}_+\Delta$ then $J$ is 
completely encoded into $\Z_4$ action on the fibers $\mathbb{S}^1$ of $\pi_+:\tilde{M}_+\to M$.) 
Conversely, given CR structure $(\hat{D},\hat{J})$ its deprolongation acts only on distribution $D$ but $\hat{J}$ does not descend. 
Thus the prolongation and de-prolongation of distributions do not have direct impact on CR structure, but help
classifying the underlying rank 2 distributions.

\section{CR structures of CR dimension 1 with maximal symmetry dimension}\label{S6}

By Theorem \ref{Th1}, if $m>3$ then $\dim\op{Sym}(M,D,J)\leq m+r$, $r\leq2$. This bound is realized 
by the algebraic model $\exp\m$, for which $\g_0$ always contains the grading element $Z$, so $r=\dim\g_0\ge1$. 

If $r=1$ then $\g_0=\R(Z)$. In this section we characterize the maximal dimension $r=2$, in which case 
$\g_0=\mathfrak{gl}(1,\C)\simeq\mathfrak{co}(2)$ as it should preserve $J$ on $\g_{-1}\simeq\R^2$.
The maximal subalgebra $\hat\g_0$ in 0th prolongation of GNLA $\m$ (without $J$) is $\mathfrak{gl}(2,\R)$.
The following statement is known, but is here for completeness.

 \begin{lemma}
The only subalgebras of $\mathfrak{gl}(2,\R)$ containing $\mathfrak{co}(2)$ are $\mathfrak{co}(2)$ and
$\mathfrak{gl}(2,\R)$ itself.
 \end{lemma}
 
 \begin{proof}
Maximal proper subalgebras of $\mathfrak{gl}(2,\R)$ are: 3D Borel $\mathfrak{b}$, $\mathfrak{sl}(2,\R)$ and $\mathfrak{co}(2)$. 
Since $\mathfrak{b}$ preserves a line in $\g_{-1}\simeq\R^2$, its only subalgebra preserving $J$ is $\R(Z)$.
We also have: $\mathfrak{sl}(2,\R)\cap\mathfrak{co}(2)=\mathfrak{so}(2)$; this holds for any copy of $\mathfrak{co}(2)$
inside $\mathfrak{gl}(2,\R)$ up to conjugation. The subalgebras of $\mathfrak{gl}(2,\R)$ of dimension 2 are: 
 $$
\text{ Cartan: } \begin{bmatrix} a & 0 \\ 0 & b\end{bmatrix}\qquad
\text{ Borel: } \begin{bmatrix} a & b \\ 0 & a\end{bmatrix}\qquad
\text{ M\"obius: } \begin{bmatrix} a & b \\ -b & a\end{bmatrix}
 $$
Of those only the latter algebra preserves a complex structure $J$.
 \end{proof}

Now we aim to classify GNLA $\m$ with $\g_0$ of one of the above two types.

 \begin{lemma}
If $\g_0\subset\op{pr}_0(\m)$ is either $\mathfrak{gl}(2,\R)$ or $\mathfrak{co}(2)$ and $\nu>3$ then $(\m,\g_0)$ is prolongation rigid.
 \end{lemma}
 
 \begin{proof}
The statement $\op{pr}(\m,\g_0)=\m\ltimes\g_0$ in the case $\g_0=\mathfrak{co}(2)$ was proved in Proposition \ref{P1}. 
Let $\g_0=\hat\g_0=\mathfrak{gl}(2,\R)$. Then $\m$ is of finite type and hence by \cite[Corollary 2.5]{Y} we have an alternative:
either $\g=\op{pr}(\m)$ is semi-simple or $\op{pr}_+(\m)=0$. In the first case $\g$ is actually a simple Lie algebra of rank 2 
with one of the following gradations: $A_2/P_1$ with $\m=\R^2$ ($\nu=1$),
$C_2/P_1$ with $\m=\mathfrak{heis}(3)$ ($\nu=2$) or $G_2/P_1$ with $\m=\m_{HC}$ ($\nu=3$).
Thus we conclude prolongation rigidity.
 \end{proof}

{\bf Case 1: $\g_0=\mathfrak{gl}(2,\R)$.} 
Each graded component $\g_{-i}$ of $\m$ is a $\g_0$ representation. Since the center $Z$ acts by scalar $-i$,
we describe $\g_{-i}$ as the direct sum of one or several $\mathfrak{sl}(2,\R)$ irreducible representations $\Gamma_t=S^t\R^2$ 
of $\dim=t+1$. The largest such $\m$ is the free GNLA (of $\nu=\infty$) and for it
 \begin{equation*}\label{free2}
\g_{-1}=\Gamma_1,\  \g_{-2}=\Gamma_0,\ \g_{-3}=\Gamma_1,\ \g_{-4}=\Gamma_2,\ \g_{-5}=\Gamma_1+\Gamma_3,\
\g_{-6}=\Gamma_0+\Gamma_2+\Gamma_4,\ \g_{-7}=2\Gamma_1+2\Gamma_3+\Gamma_5,\ \dots
 \end{equation*}
These equalities are obtained from the inclusion $\g_{-i-1}\subset\g_{-1}\otimes\g_{-i}$ and the Jacobi identity.
The growth vector of this free GNLA is derived from Hall bases and the M\"obius inversion formula \cite{S}:
 $$
2, 1, 2, 3, 6, 9, 18, 30, 56, 99, 186, 335, 630, 1161, 2182, 4080, 7710, 14532, 27594, 52377, \dots 
 $$
Finite depth free GNLA is obtained from infinite free GNLA by truncation after the grading $-\nu$.

Any other GNLA with $\g_0=\op{pr}_0(\m)=\mathfrak{gl}(2,\R)$ is obtained from the free GNLA as a quotient by a graded ideal.
In terms of components, this means a reduction of one or several modules $\Gamma_t$ in $\g_{-i}$, 
compatible with bracket generation and the Jacobi identity. In \cite{K3} we call such GNLA sub-free Lie algebras.

For example, $\g_{-5}$ is the direct sum of two proper submodules $\Gamma_1$ and $\Gamma_3$.
These can be taken as reductions $\g_{-5}'$ and $\g_{-5}''$, respectively, giving two sub-free GNLA of depth 5: 
 $$
\m'_5= \g_{-5}'\oplus\g_{-4}\oplus\g_{-3}\oplus\g_{-2}\oplus\g_{-1}\quad\text{and}\quad
\m''_5= \g_{-5}''\oplus\g_{-4}\oplus\g_{-3}\oplus\g_{-2}\oplus\g_{-1}.
 $$
These have growth $(2,1,2,3,2)$ and $(2,1,2,3,4)$, respectively. For instance, $\m'_5$ is given by the structure relations
(where $e^j_i\in\g_{-i}$)
 $$
[e'_1,e''_1]=e_2,\ [e'_1,e_2]=e'_3,\ [e''_1,e_2]=e''_3,\ [e'_1,e'_3]=e'_4,\ 
[e'_1,e''_3]=[e''_1,e'_3]=e''_4,\ [e''_1,e''_3]=e'''_4.
 $$Choosing one of those sub-free GNLA or the free truncated GNLA of depth 5 we consider a submodule of $\g_{-6}$ 
that is bracket generated by those, possibly reduce some of irreps, 
 $$
\m'_6= \g_{-6}'\oplus\m'_5\quad\text{and}\quad\m''_6= \g_{-6}''\oplus\m'_5,
 $$
where $\g_{-6}'\subset\Gamma_0+\Gamma_2$ and $\g_{-6}''\subset\Gamma_2+\Gamma_4$, etc.

\medskip

{\bf Case 2: $\g_0=\mathfrak{gl}(1,\C)$.} 
If the growth vector starts $(2,1,1,\dots)$ then $\g_0=\op{pr}_0(\m)\cap\mathfrak{gl}(\g_{-1},J)$ is $\R(Z)$.
Thus we consider sub-free GNLA, with the growth vector  $(2,1,2,\dots)$,
perform a reduction to $\mathfrak{gl}(1,\C)\subset\mathfrak{gl}(2,\R)$ given by $J$ on $\g_{-1}$, 
and then continue graded central extensions keeping $\mathfrak{co}(2)$ invariance.

Technically this means that extensions are encoded via $\mathfrak{co}(2)$ orbits in $\op{Gr}_d(H^2(\m)_{\nu+1})$.
Here are the first low-dimensional cases: rank 2 distributions with $\g_0=\mathfrak{gl}(1,\C)$,
examples are elaborated upon \cite[\S5.3]{AK}. 

We start with 5D growth $(2,1,2)$ and compute successive extensions.
The structure relations of $\m_{HC}$ are:
 \begin{equation}\label{HC}
[e'_1,e''_1]=e_2,\ [e'_1,e_2]=e'_3,\ [e''_1,e_2]=e''_3.
 \end{equation}

6D: $(2,1,2,1)$. This $\mathfrak{ell}_6$ is an extension of $\m_{HC}$.
The additional relations to \eqref{HC} are
 \begin{equation}\label{ell6}
[e'_1,e'_3]=e_4,\ [e''_1,e''_3]=e_4.
 \end{equation}

7D: $(2,1,2,2)$. This $\mathfrak{ell}_7$ is an extension of $\m_{HC}$.
The additional relations to \eqref{HC} are
 \begin{equation}\label{ell7}
[e'_1,e'_3]=-[e''_1,e''_3]=e'_4,\ [e'_1,e''_3]=[e''_1,e'_3]=e''_4.
 \end{equation}

8D: $(2,1,2,1,2)$. This $\mathfrak{ell}_8$ is an extension of $\mathfrak{ell}_6$.
The additional relations to \eqref{HC}-\eqref{ell6} are
 \begin{equation}\label{ell8}
[e'_1,e_4]=[e_2,e''_3]=e'_5,\ [e''_1,e_4]=-[e_2,e'_3]=e''_5.
 \end{equation}
In all the formulae $e_i^j$ is a basis in $\g_{-i}$, and one can extend indefinitely. 

We summarize the above results, using Corollary \ref{Tanaka-flat}, as follows:

 \begin{cor}
Every CR structure $(D,J)$ of CR dimension 1 on a manifold $M$ of dimension $m\ge5$ with $\dim\mathfrak{sym}(M,D,J)=m+2$
is locally isomorphic to $\exp\m$. The growth vector of the distribution starts $(2,1,2,\dots)$ and $\m$ is either
a sub-free GNLA with $\hat{\g}_0=\mathfrak{gl}(2,\R)$ or a reduction of it with $\hat{\g}_0=\mathfrak{gl}(1,\C)$ 
and further sequence of $\mathfrak{co}(2)$-invariant graded central extensions. 
Up to conjugation, the structure $J$ is unique
when $\dim\hat{\g}_0=4$ or is determined up to sign $\pm J$ when $\dim\hat{\g}_0=2$.
 \end{cor}

\section{Examples of coordinate models}\label{S7}

Let us now consider a few coordinate realizations; some of these models were previously obtained and studied in \cite{Sh},
with further elaboration in \cite{BES,MS}.
The computations have been verified in Maple, which is supplemented to the arXiv submission.

The two models of CR dimension 1 with depth 3, both extending the classical 3D sphere $\op{Im}(u)=|z|^2$, are
respectively $M^4\subset\C^3$ and $M^5\subset\C^4$ given by the equations based on classical distributions:
 \begin{align}
\text{``Engel'':}\qquad &\op{Im}(u)=|z|^2,\quad \op{Im}(v)=\op{Re}\bigl(z^2\bar{z}\bigr);\label{ENG}\\
\text{``Hilbert-Cartan'':}\qquad &\op{Im}(u)=|z|^2,\quad \op{Im}(v)=\op{Re}\bigl(z^2\bar{z}\bigr),\quad
 \op{Im}(w)=\op{Im}\bigl(z^2\bar{z}\bigr).\label{CAR}
 \end{align}
They are the base of other models, as stated in Theorem \ref{Th3}.
Their symmetry algebra is given by Theorem \ref{Th1} with $r=1,2$ respectively.
(All computations are supported with independent Maple verifications.)

Next, consider the following model $M^6\subset\C^5$ of reduced growth $(2,1,2,1)$:
 \begin{equation}\label{2121}
\op{Im}(u)=|z|^2,\quad \op{Im}(v)=\op{Re}\bigl(z^2\bar{z}\bigr),\quad \op{Im}(w)=\op{Im}\bigl(z^2\bar{z}\bigr),\quad 
\op{Im}(s)=\op{Re}\bigl(z^3\bar{z}\bigr)+a|z|^4. 
 \end{equation}
The distribution has type elliptic $\mathfrak{ell}_6$ for $|a|>\tfrac34$, parabolic $\mathfrak{p}_6$ for $|a|=\tfrac34$
and hyperbolic $\mathfrak{h}_6$ for $|a|<\tfrac34$ (including $\infty$ with renormalization).
Symmetry of \eqref{2121} is given by Theorem \ref{Th1} with $r=1$; the only exception arises for
one parameter value in the elliptic case. Indeed, there exists a unique (up to automorphism)
complex structure $J$ on the distribution $\mathfrak{p}_6$, while a one-parametric family of such for both $\mathfrak{h}_6$  
and $\mathfrak{ell}_6$. In the latter case only one choice of $\pm J$ is invariant wrt $\g_0$ and this gives the special
value of parameter $a=\infty$, the case missed in \cite{Sh} (by renormalization, the last equation of \eqref{2121} takes the form
$\op{Im}(s)=|z|^4$) so only this case gives $r=2$. 
The symmetries of \eqref{2121} are explicitly given by\footnote{The real symmetries are given by $\op{Re}(S_k)$ indeed, but we indicate the holomorphic part for clarity.} 
 \begin{align*}
S_0 =\ & z\p_z+2u\p_u+3v\p_v+3w\p_w+4s\p_s,\\
S'_1 =\ & \p_z +2iz\p_u +\bigl(2u+iz^2\bigr)\p_v +z^2\p_w +\bigl(v(4a+3)+iz^3\bigr)\p_s,\\
S''_1 =\ & i\p_z +2z\p_u +z^2\p_v + \bigl(2u-iz^2\bigr)\p_w+\bigl(w(4a-3)+z^3\bigr)\p_s,\\
S_2 =\ & \p_u,\qquad  S'_3 = \p_v,\qquad  S''_3 = \p_w,\qquad  S_4 = \p_s,
 \end{align*}
together with special symmetry $S_{0J}=iz\p_z+v\p_w-w\p_v$ valid only for $a=\infty$
(and renormalization of generators $S_1'$, $S_1''$ so that the last terms are $4v\p_s$ and $4w\p_s$, respectively).

\smallskip

Next, consider the following model $M^7\subset\C^6$  of reduced growth $(2,1,2,2)$:
 \begin{equation}\label{2122}
 \begin{split}
& \op{Im}(u)=|z|^2,\quad \op{Im}(v)=\op{Re}(z^2\bar{z}),\quad \op{Im}(w)=\op{Im}(z^2\bar{z}),\\ 
& \op{Im}(s)=\op{Re}(z^3\bar{z})+a|z|^4,\quad \op{Im}(t)=\op{Im}(z^3\bar{z})+b|z|^4. 
 \end{split}
 \end{equation}
The distribution has type elliptic $\mathfrak{ell}_7$ for $a^2+b^2<\tfrac9{16}$, parabolic $\mathfrak{p}_7$ for 
$a^2+b^2=\tfrac9{16}$ and hyperbolic $\mathfrak{h}_7$ for $a^2+b^2>\tfrac9{16}$ in the notation of \cite[\S5.3]{AK}.
Symmetry of \eqref{2122} is given by Theorem \ref{Th1} with $r=1$, however there is again one exception for the 
special value of parameter (though we use two parameters $a,b$ for symmetry, the complex structure $J$ on $\g_{-1}$
depends on one real parameter) namely $a=b=0$. 

The symmetries of \eqref{2122} are explicitly given by
 \begin{align*}
S_0 =\ & z\p_z+2u\p_u+3v\p_v+3w\p_w+4s\p_s+4t\p_t,\\
S'_1 =\ & \p_z +2iz\p_u +\bigl(2u+iz^2\bigr)\p_v +z^2\p_w +\bigl(v(4a+3)+iz^3\bigr)\p_s
+\bigl(4bv+3w+z^3\bigr)\p_t,\\
S''_1 =\ & i\p_z +2z\p_u +z^2\p_v + \bigl(2u-iz^2\bigr)\p_w+\bigl(w(4a-3)+z^3\bigr)\p_s
+\bigl(4bw+3v-iz^3\bigr)\p_t,\\
S_2 =\ & \p_u,\qquad  S'_3 = \p_v,\qquad  S''_3 = \p_w,\qquad  S'_4 = \p_s,\qquad  S''_4 = \p_t,
 \end{align*}
together with special symmetry $S_{0J}=iz\p_z+v\p_w-w\p_v+2s\p_t-2t\p_s$ valid only for $a=b=0$
(and evaluation of generators $S_1'$, $S_1''$ on that).
Note that we can also consider limiting case $(a,b)\to\infty$ and get another $\mathfrak{h}_7$ case,
with obvious modifications for the equations and symmetries.

Next consider $M^8\subset\C^7$ with reduced growth $(2,1,2,3)$:
 \begin{equation}\label{2123}
 \begin{split}
& \op{Im}(u)=|z|^2,\quad\ \op{Im}(v)=\op{Re}(z^2\bar{z}),\quad\ \op{Im}(w)=\op{Im}(z^2\bar{z}),\\ 
& \op{Im}(s)=\op{Re}(z^3\bar{z}),\quad\ \op{Im}(t)=\op{Im}(z^3\bar{z}),\quad\ \op{Im}(q)=|z|^4. 
 \end{split}
 \end{equation}
This is the free GNLA of depth 4 hence $\hat{\g}_0=\op{pr}_0(\m)=\mathfrak{gl}(2,\R)$ but $\g_0=\mathfrak{gl}(1,\C)$, 
so here we have $r=2$ again. The symmetry of \eqref{2123} is given by Theorem \ref{Th1} with $r=2$. Explicitly, we have:
 \begin{align*}
S'_0 =\ & z\p_z+2u\p_u+3v\p_v+3w\p_w+4s\p_s+4t\p_t+4q\p_q,\\
S''_0 =\ & iz\p_z-w\p_v+v\p_w-2t\p_s+2s\p_t,\\
S'_1 =\ & \p_z +2iz\p_u +\bigl(2u+iz^2\bigr)\p_v +z^2\p_w +\bigl(3v+iz^3\bigr)\p_s
+\bigl(3w+z^3\bigr)\p_t+4v\p_q,\\
S''_1 =\ & i\p_z +2z\p_u +z^2\p_v + \bigl(2u-iz^2\bigr)\p_w+\bigl(-3w+z^3\bigr)\p_s
+\bigl(3v-iz^3\bigr)\p_t+4w\p_q,\\
S_2 =\ & \p_u,\qquad  S'_3 = \p_v,\qquad  S''_3 = \p_w,\qquad  S'_4 = \p_s,\qquad  S''_4 = \p_t,\qquad  S'''_4 = \p_q.
 \end{align*}

Finally, let us realize the model $\mathfrak{ell}_8$ given by \eqref{ell8} by another $M^8\subset\C^7$ that is not rigid:
 \begin{equation}\label{21212}
 \begin{split}
& \op{Im}(u)=|z|^2,\quad\ \op{Im}(v)=\op{Re}(z^2\bar{z}),\quad\ \op{Im}(w)=\op{Im}(z^2\bar{z}),\\ 
& \op{Im}(q)=|z|^4,\quad\ \op{Im}(s)=\op{Re}(z^2\bar{z})\op{Re}(u),\quad\ \op{Im}(t)=\op{Im}(z^2\bar{z})\op{Re}(u). 
 \end{split}
 \end{equation}
The symmetry of \eqref{21212} is given by Theorem \ref{Th1} with $r=2$. Explicitly, we have:
 \begin{align*}
S'_0 =\ & z\p_z+2u\p_u+3v\p_v+3w\p_w+4q\p_q+5s\p_s+5t\p_t,\\
S''_0 =\ & iz\p_z-w\p_v+v\p_w-t\p_s+s\p_t,\\
S'_1 =\ & \p_z +2iz\p_u +\bigl(2u+iz^2\bigr)\p_v +z^2\p_w +4v\p_q
+\bigl(u^2+iz^2u\bigr)\p_s+\bigl(z^2u-q\bigr)\p_t,\\
S''_1 =\ & i\p_z +2z\p_u +z^2\p_v + \bigl(2u-iz^2\bigr)\p_w+4w\p_q
+\bigl(z^2u+q\bigr)\p_s+\bigl(u^2-iz^2u\bigr)\p_t,\\
S_2 =\ & \p_u+v\p_s+w\p_t,\quad\  S'_3 = \p_v,\quad\  S''_3 = \p_w,\quad\  
S_4 = \p_q,\quad\  S'_5 = \p_s,\quad\  S''_5 = \p_t.
 \end{align*}


 \begin{rk}\label{Rk2}
Let us note that model \eqref{ENG} as well as the general Goursat model \eqref{GOU} is tube,
while models \eqref{CAR}-\eqref{2123} are non-tubular but rigid. Indeed, the maximal dimension of the
Abelian symmetry algebra in those cases is $m-2$ where $m=\dim M$. This already fails for models of
larger depth, for instance model \eqref{ell8} has Abelian symmetry algebra of maximal dimension 5, hence is not rigid.

However, CR models supported on distributions (4) of \cite{AK} are rigid (though not tubes). 
These models correspond to Monge equations\footnote{Monge equations are underdetermined systems of ODEs, possibly of
different orders.} ($n\ge m$)
 $$
y^{(m)}(x)=\bigl(z^{(n)}(x)\bigr)^2 
 $$
encoded as a rank 2 distribution $D$ on $M^{m+n+2}$. The symmetry algebra possesses an Abelian subalgebra of 
dimension $m+n$, given through GNLA $\m$ as $\oplus_{i>1}\g_{-i}$, see \cite[Proposition 7]{AK}.
For these distributions $\g_0=\op{pr}(\m)$ is either Borel or Cartan, hence they support 1-parametric family of
left-invariant $J$ on $D=\g_{-1}$ and the symmetry of the corresponding CR structure has $r=1$ in terms of Theorem \ref{Th1}.
These CR structures can be encoded similar to considered models \eqref{2121}-\eqref{2123}, which should be
of degree equal to depth, i.e.\ $n+1+\delta_{m,n}$. 
 \end{rk}

Following this remark for $n=2$, $m=1$ we get the Hilbert-Cartan equation $y'(x)=z''(x)^2$ whose equation-manifold  
$\R^5(x,y_0,z_0,z_1,z_2)\subset J^{1,2}(\R,\R^{1+1})$ in the mixed jet-space carries the (2,3,5) distribution. 
It underlines CR structure \eqref{CAR} and corresponds to the generalized flag variety $G_2/P_1$ (as a local chart).

\medskip

As the very last example, let us consider the full flag variety of the exceptional Lie group $G_2$, namely $G_2/P_{12}$,
where the maximal parabolic $P_{12}=B$ is the Borel subgroup. 
The corresponding Monge equation (a local chart in the above complete flag variety) is given by partial prolongation:
 $$
y'(x)=z''(x)^2,\quad y''(x)=2z''(x)z'''(x). 
 $$
Note that the equation-manifold $\R^6(x,y_0,z_0,z_1,z_2,z_3)\subset J^{2,3}(\R,\R^{1+1})$ in the mixed jet-space carries 
the distribution $D$ spanned by $\mathcal{D}_x=\p_x+z_2^2\p_{y_0}+z_1\p_{z_0}+z_2\p_{z_1}+z_3\p_{z_2}$ and $\p_{z_3}$.
There is a one-parameter group of invariant complex structures on this rank 2 distribution
that (after rescalings) is given by $J(\mathcal{D}_x)=\p_{z_3}+t\,\mathcal{D}_x$ for a parameter $t\in\R$.
Its symmetry is governed by Theorem \ref{Th1} with $r=1$, so it is 7-dimensional.

Realization of this CR structure as a real submanifold $M^6\subset\C^5$ is not that straightforward, as central extensions 
\eqref{2121} and \eqref{2122} of \eqref{CAR}, since this is the prolongation of the underlying distribution only. 
This 6-dimensional model is not a tube, characterized by the existence of a 5-dimensional Abelian subalgebra of symmetries, 
neither it is rigid, which would posses a 4-dimensional Abelian subalgebra of symmetries, 
yet it has a 3-dimensional Abelian subalgebra of symmetries, and those can be simultaneously straightened.
A computation based on this idea yields the following model depending on a real parameter $\epsilon$: 
 \begin{equation}\label{G2B}
 \begin{split}
& \op{Im}(s) =|z|^2,\quad
\op{Im}(u) =\op{Re}\bigl(zs+z^2\bar{z}\bigr),\quad
\op{Im}(v) =\op{Re}\bigl(z^2s+z^3\bar{z}\bigr)+\tfrac34|z|^4,\\
& \quad\ \op{Im}(w) =\op{Re}\bigl(z^3s+24z^2\bar{z}s-2(\epsilon+i) z^4\bar{z}-4(\epsilon+5i)z^3\bar{z}^2\bigr).
 \end{split}
 \end{equation}
The symmetries are again governed by Theorem \ref{Th1} with $r=1$ and here they are in explicit form:
 \begin{align*}
S_0 =\ & z\p_z+2s\p_s+3u\p_u+4v\p_v+5w\p_w,\\
S'_1 =\ & \p_z +2iz\p_s +\bigl((2+i)s-(2-i)z^2\bigr)\p_u +\bigl(6u-4izs-(2-i)z^3\bigr)\p_v\\
 &+(24s^2 -16\epsilon v  +(27+16\epsilon)iz^2s   -2i\epsilon z^4)\p_w,\\
 S''_1 =\ & i\p_z +2z\p_s +\bigl((1+2i)z^2-s\bigr)\p_u + \bigl((1+2i)z^2 - 2s\bigr)z\p_v\\
 &+\bigl(16v +(21-16i)z^2s -2\epsilon z^4\bigr)\p_w,\\
 S_2 =\ & \p_s +iz\p_u +iz^2\p_v +\bigl(24u -24izs +iz^3\bigr)\p_w,\\
 S_3 =\ & \p_u,\qquad  S_4 = \p_v,\qquad  S_5 = \p_w.
 \end{align*}

The structure of this symmetry algebra is $\mathfrak{s}=\R\ltimes\mathfrak{n}$, here $\mathfrak{n}$ is the nilpotent approximation
of the distribution $D$, which is also the nilradical of the Borel subalgebra $\mathfrak{b}\subset\op{Lie}(G_2)$.
This 7-dimensional solvable Lie algebra $\mathfrak{s}$ differs from the 7-dimensional solvable symmetry algebra of \eqref{CAR}
because one case is obtained from the other by prolongation of the distribution but not of the CR structure.
Both algebras are 2-dimensional extensions of the 5-dimensional GNLA $\m_{HC}$ of growth $(2,1,2)$, namely by a
2-dimensional subalgebra of $\op{pr}_0(\m)=\mathfrak{gl}(2,\R)$. 
While for \eqref{CAR} this is an extension by $\mathfrak{gl}(1,\C)\subset\mathfrak{gl}(2,\R)$,
in case of \eqref{G2B} it is an extension by $\mathfrak{sol}(2)\subset\mathfrak{gl}(2,\R)$ with the following
re-grading of the Lie algebra.

This gives a coordinate realization of the first non-Goursat homogeneous CR model from Theorem \ref{Th2}.

\smallskip

All these examples demonstrate how the models behave under integrable extension and under 
prolongation of the underlying distribution: while in the former case the formulae and the symmetries are inherited,
the latter case is rather involved. And indeed, the nonrigid models of Theorem \ref{Th2} are inherited from the Goursat 
distribution, but not with respect to each other as is apparent from the proof of Proposition \ref{P2}.

\medskip

{\bf Acknowledgment.} 
I would like to thank Valery Beloshapka, Jan Gregorovi\v{c} and David Sykes for useful discussions, 
which happened during the conference on $k$-nondegenerate CR structures in SUSTech University, Shenzhen, China. 
I am grateful to the main organizer Ilya Kossovsky for hospitality and inspiring discussion sessions. 
Special thanks to an anonymous referee for a thoughtful review.
This research was supported by the UiT Aurora project MASCOT, TFS project Lie-St\o rmer Center and 
NFR project MatNorVeg.


\end{document}